\documentclass[12pt]{amsart}
\usepackage{url}
\usepackage{epsfig}
\usepackage[all]{xy}
\usepackage{graphicx}
\usepackage{mathrsfs}
\usepackage{multirow}

\usepackage{color}
\newcommand{\tony}[1]{{\color{blue} \sf  Tony: [#1]}}

\newcommand{\defi}[1]{\textsf{#1}} % for defined terms

\newcommand{\bA}{{\mathbb A}}
\newcommand{\bC}{{\mathbb C}}
\newcommand{\bF}{{\mathbb F}}
\newcommand{\bG}{{\mathbb G}}

\newcommand{\bP}{{\mathbb P}}
\newcommand{\bQ}{{\mathbb Q}}
\newcommand{\bR}{{\mathbb R}}
\newcommand{\bZ}{{\mathbb Z}}

\newcommand{\kbar}{{\overline{k}}}

\newcommand{\Adeles}{{\mathbf A}}
\newcommand{\kk}{{\mathbf k}}

\newcommand{\pp}{{\mathfrak p}}

\newcommand{\calA}{{\mathcal A}}

\newcommand{\calC}{{\mathcal C}}

\newcommand{\calF}{{\mathcal F}}

\newcommand{\calO}{{\mathcal O}}

\newcommand{\calS}{{\mathcal S}}

\newcommand{\calW}{{\mathcal W}}
\newcommand{\calX}{{\mathcal X}}
\newcommand{\calY}{{\mathcal Y}}
\newcommand{\calZ}{{\mathcal Z}}

\newcommand{\sO}{{\mathscr O}}

\DeclareMathOperator{\rk}{rk}

\DeclareMathOperator{\inv}{inv}

\DeclareMathOperator{\im}{im}

\DeclareMathOperator{\Br}{Br}

\DeclareMathOperator{\Cl}{Cl}

\DeclareMathOperator{\Pic}{Pic}

\DeclareMathOperator{\Spec}{Spec}

\DeclareMathOperator{\Proj}{Proj}

\DeclareMathOperator{\ev}{ev}
\DeclareMathOperator{\res}{res}
\DeclareMathOperator{\Mat}{Mat}
\DeclareMathOperator{\sm}{sm}

\newcommand{\et}{{\rm \mathaccent 19 et}}

\newcommand{\ra}{\rightarrow}

\newtheorem{theorem}{Theorem}[section]
\newtheorem{lemma}[theorem]{Lemma}
\newtheorem{corollary}[theorem]{Corollary}
\newtheorem{proposition}[theorem]{Proposition}

\theoremstyle{definition}

\theoremstyle{remark}
\newtheorem{remark}[theorem]{Remark}
\newtheorem{remarks}[theorem]{Remarks}

\begin{document}

\title{Failure of the Hasse principle on general K3 surfaces}
%\keywords{9999}
\author{Brendan Hassett}
\thanks{This research was supported by National Science Foundation Grants 0554491, 0901645, and 1103659.}
\address{Department of Mathematics, Rice University, Houston, TX 77005, USA}
\email{hassett@rice.edu}
\urladdr{http://www.math.rice.edu/\~{}hassett}

\author{Anthony V\'arilly-Alvarado}
\address{Department of Mathematics, Rice University, Houston, TX 77005, USA}
\email{varilly@rice.edu}
\urladdr{http://www.math.rice.edu/\~{}av15}

\begin{abstract}
We show that transcendental elements of the Brauer group of an algebraic surface can obstruct the Hasse principle. We construct a general K3 surface $X$ of degree $2$ over $\bQ$, together with a two-torsion Brauer class $\alpha$ that is unramified at every finite prime, but ramifies at real points of $X$. Motivated by Hodge theory, the pair $(X,\alpha)$ is constructed from a double cover of $\bP^2\times\bP^2$ ramified over a hypersurface of bi-degree $(2,2)$.
\end{abstract}

\maketitle

%****************************************************************************
\section{Introduction}

Let $X$ be a smooth projective geometrically integral variety over a number field $k$.  If $X$ has a $k_v$-point for every place $v$ of $k$ (equivalently, if its set $X(\Adeles)$ of adelic points is nonempty), yet it does not have a $k$-point, then we say that $X$ does not satisfy the \defi{Hasse principle}.  Manin~\cite{ManinICM} showed that any subset $\calS$ of the Brauer group $\Br(X) := H^2_\et(X,\bG_m)$ may be used to construct an intermediate set
\[
X(k) \subseteq X(\Adeles)^{\calS} \subseteq X(\Adeles)
\]
that often explains failures of the Hasse principle, in the sense that $X(\Adeles)^{\calS}$ may be empty, even if $X(\Adeles)$ is not.  In this case, we say there is a \defi{Brauer-Manin obstruction} to the Hasse principle for $X$. See~\S\ref{s: local invariants} for the definition of $X(\Adeles)^{\calS}$.

There is a filtration of the Brauer group $\Br_0(X) \subseteq \Br_1(X) \subseteq \Br(X)$, where
\begin{align*}
\Br_0(X) &:= \im\left( \Br(k) \to \Br(X) \right), \\
\Br_1(X) &:= \ker\left( \Br(X) \to \Br(\overline{X})\right),
\end{align*}
and $\overline{X} = X\times_k \kbar$ for a fixed algebraic closure $\kbar$ of $k$.   Elements in $\Br_0(X)$ are said to be \defi{constant}; class field theory shows that if $\calS \subseteq \Br_0(X)$, then $X(\Adeles)^\calS = X(\Adeles)$, so these elements cannot obstruct the Hasse principle.  Elements in  $\Br_1(X)$ are called \defi{algebraic}; the remaining elements of the Brauer group are \defi{transcendental}.  

There is a large body of literature, spanning the last four decades, on algebraic Brauer classes and algebraic Brauer-Manin obstructions to the Hasse principle and the related notion of weak approximation (i.e., where sets $\calS \subseteq \Br_1(X)$ suffice to explain failures of these phenomena); see, for example \cite{ManinCubic,BSD,CTCS80,CTSSD87,CTKS87,SD93,SD99,KT04,Bright02,BBFL07,Cunnane,Corn,KT08,Logan,VA08,LoganVanLuijk,EJ-CJM,EJ-JCNT,Corn10,EJ-IJNT}. The systematic study of these obstructions benefits in no small part from an isomorphism
\begin{equation}
	\label{eq: HS}
	\Br_1(X)/\Br_0(X) \xrightarrow{\sim} H^1\big(k,\Pic(\overline{X})\big),
\end{equation}
coming from the Hochschild-Serre spectral sequence.

Obstructions arising from transcendental elements, on the other hand, remain mysterious,  because it is difficult to get a concrete handle on transcendental elements of the Brauer group; there is no known analogue of~\eqref{eq: HS} for the group $\Br(X)/\Br_1(X)$.

If $X$ is a curve, or a surface of negative Kodaira dimension, then $\Br\big(\overline{X}\big) = 0$, so the Brauer group is entirely algebraic. On the other hand, in 1996 Harari constructed a $3$-fold with a transcendental Brauer-Manin obstruction to the Hasse principle \cite{Harari}.  This begs the question: what about algebraic surfaces? Can transcendental Brauer classes obstruct the Hasse principle on an algebraic surface? A natural place to study this question is the class of K3 surfaces; they are arguably some of the simplest surfaces of nonnegative Kodaira dimension in the Castelnuovo-Enriques-Manin classification.  The group $\Br(X)/\Br_1(X)$ is finite for a K3 surface~\cite{SkorobogatovZarhin}, but it can be nontrivial.

With arithmetic applications in mind, several authors over the last decade have constructed explicit transcendental elements on K3 surfaces \cite{Wittenberg,SkorobogatovSwinnertonDyer,HarariSkorobogatov,Ieronymou,Preu,IeronymouSkorobogatovZarhin,SkorobogatovZarhin2}.  Wittenberg, Ieronymou and Preu have used these elements to exhibit obstructions  \defi{weak approximation} (i.e., density of $X(k)$ in $\prod_v X(k_v)$ for the product of the $v$-adic topologies).  In all cases the K3 surfaces considered have elliptic fibrations that play a vital role in the construction of transcendental classes.

%Motivated to construct trancendental classes on \emph{general} K3 surfaces (i.e., surfaces with geometric Picard number 1), and inspired by Hodge-theoretic work of van Geemen and Voisin \tony{reference}, we constructed a pair $(X,\alpha)$ consisting of a K3 surface

Inspired by Hodge-theoretic work of van Geemen and Voisin~\cite{vanGeemen,Voisin}, we recently constructed a K3 surface with geometric Picard number $1$ (and hence no elliptic fibrations), together with a transcendental Brauer class $\alpha$ obstructing weak approximation; see~\cite{HVAV} (joint with Varilly).  The pair $(X,\alpha)$ was obtained from a cubic fourfold containing a plane.  At the time, we were unable to extend our work to obtain counterexamples to the Hasse principle, in part because we were unable to control the invariants of $\alpha$ at real points of $X$---ironically, this is precisely the reason we obtain a counterexample to weak approximation! See Remarks~\ref{rems: computations} as well.

Taking advantage of some recent developments (see Remarks~\ref{rems: computations}), our goal in this paper is to rectify the above situation and show, once and for all, that transcendental Brauer classes on algebraic surfaces can obstruct the Hasse principle. 

\begin{theorem}
\label{thm: main}
Let $X$ be a K3 surface of degree $2$ over a number field $k$, with function field $\kk(X)$, given as a sextic in the weighted projective space $\bP(1,1,1,3) = \Proj k[x_0,x_1,x_2,w]$ of the form
\begin{equation}
\label{eq: explicit K3}
w^2 = -\frac{1}{2}\cdot\det
\begin{pmatrix}
2A & B & C \\
B & 2D & E \\
C & E & 2F
\end{pmatrix},
\end{equation}
where $A,\dots,F \in k[x_0,x_1,x_2]$ are homogeneous quadratic polynomials. Then the class $\calA$ of the quaternion algebra $(B^2 - 4AD,A)$ in $\Br(\kk(X))$ extends to an element of $\Br(X)$.

When $k = \bQ$, there exist particular polynomials $A,\dots,F \in \bZ[x_0,x_1,x_2]$ such that $X$ has geometric Picard rank $1$ and $\calA$ gives rise to a transcendental Brauer-Manin obstruction to the Hasse principle on $X$.
\end{theorem}

\begin{remark}
\label{rem: vG summary}
In~\cite[\S9]{vanGeemen}, Van Geemen showed that \emph{every} Brauer class $\alpha$ of order $2$ on a polarized complex K3 surface $(X,f)$ of degree $2$ with $\Pic(X) = \bZ f$ gives rise to (and must arise from) one of three types of varieties: 
\begin{itemize}
\item a smooth complete intersection of three quadrics in $\bP^5$ (itself a K3 surface), or 
\item a cubic fourfold containing a plane, or 
\item a double cover of $\bP^2\times \bP^2$ ramified along a hypersurface of bidegree $(2,2)$. 
\end{itemize}
More precisely, the class $\alpha$ determines a sublattice $T_\alpha \subseteq T_X$ of the transcendental lattice of $X$ which is a polarized Hodge structure, a twist of which is Hodge isometric to a primitive sublattice of the middle cohomology of one the three types of varieties above. See \S\ref{s: background} for more details. 

The Azumaya algebra $\calA$ of Theorem~\ref{thm: main} represents a class arising from a double cover of $\bP^2\times \bP^2$ ramified along a hypersurface of bidegree $(2,2)$.
\end{remark}

\begin{remarks} 
\label{rems: computations}
We record a few remarks on the computational subtleties behind the second part of Theorem~\ref{thm: main}:
\begin{enumerate}
\item For computational purposes, we go in a direction ``opposite'' to van Geemen: starting from one of the three types of varieties described in Remark~\ref{rem: vG summary}, defined over a number field $k$, we recover a K3 surface $X$ over $k$ of degree $2$, together with a $2$-torsion Azumaya algebra $\calA$. Unfortunately, there is no guarantee that $X$ has geometric Picard number $\rho = 1$; in fact, it need not.  We use a recent theorem of Elsenhans and Jahnel~\cite{ElsenhansJahnelOnePrime} to certify that our example has $\rho = 1$.
\item Curiously, one of the most delicate steps in the proof of Theorem~\ref{thm: main} is determining the primes of bad reduction of $X$.  We have to factor an integer with 318 decimal digits, whose smallest prime factor turns out to have 66 digits!
\item We use some of our work on varieties parametrizing maximal isotropic subspaces of families of quadrics admitting at worst isolated singularities to show that $\calA$ can ramify only at the real place, $2$-adic places and primes of bad reduction for $X$~\cite[\S3]{HVAV}.  These are thus  the only places where the local invariants of $\calA$ can be nontrivial.
\item We rely on recent work of Colliot-Th\'el\`ene and Skorobogatov~\cite{CTS-TAMS} to control the local invariants for the algebra $\calA$ at odd primes of bad reduction.
\end{enumerate}
\end{remarks}

\begin{remark}
The Azumaya algebra of Theorem~\ref{thm: main} looks remarkably similar to the algebra we used in~\cite{HVAV} to exhibit counter-examples to weak approximation.  This is not a coincidence: compare Theorem~\ref{thm: Stein} with~\cite[Theorem~5.1]{HVAV}.
\end{remark}

%\tony{connections with derived categories?}

\subsection*{Outline of the paper}
In~\S\ref{s: background} we explain the content of Remark~\ref{rem: vG summary} in detail, following van Geemen~\cite{vanGeemen}. The section is not logically necessary for the paper, but we include it for completeness because it explains how to construct, in principle, Azumaya algebras representing every two-torsion Brauer class on a general K3 surface of degree $2$. 

In~\S\ref{s: unramified}, we explain how to explicitly construct, from a double cover of $\bP^2\times \bP^2$ ramified along a hypersurface of bidegree $(2,2)$, a pair $(X,\alpha)$ where $X$ a K3 surface of degree $2$ and $\alpha \in \Br(X)[2]$ is an Azumaya algebra. We work mostly over a discrete valuation ring (see Theorem~\ref{thm: Stein}). This flexibility later affords us control, when working over number fields, of local invariants at places where $\alpha$ ramifies; see Lemma~\ref{lem:most invariants}. In~\S\ref{s: local invariants}, we give a collection of sufficient conditions to control the evaluation maps of $\alpha$ over number fields, specializing ultimately to the case $k = \bQ$. Notably, Proposition~\ref{prop:constant eval map} (due to Colliot-Th\'el\`ene and Skorobogatov) together with Lemma~\ref{lem:no cyclic covers} show that the evaluation maps of $\alpha$ are constant at non-2-adic finite places of bad reduction of $X$ whenever the singular locus consists of $r < 8$ ordinary double points.

We use this preparatory work to give an example in~\S\ref{s: an example} of a surface witnessing the second part of Theorem~\ref{thm: main}. In~\S\ref{s: computations} we give details of how we found the example of~\S\ref{s: an example}, using a computer.

\subsection*{Acknowledgements} We thank Olivier Wittenberg for helpful comments and correspondence on transcendental Brauer classes on general K3 surfaces of degree $2$~\cite{WittenbergLetter}. Our computations were carried out using Macaulay~2~\cite{M2}, {\tt Magma}~\cite{BCP} and SAGE~\cite{sage}

%****************************************************************************
\section{Lattices and Hodge theory}
\label{s: background}

In this section, all varieties are defined over $\bC$.  Our goal here is to outline van Geemen's geometric constructions representing two-torsion Brauer classes on a K3 of degree $2$ and Picard rank $1$.  Strictly speaking, this section is not logically necessary in the proof of Theorem~\ref{thm: main}, and we use only one of the three constructions described. We include it, however, so that readers not acquainted with these ideas get a clear sense of the geometric motivation behind Theorem~\ref{thm: main}.

Let $X$ be a complex K3 surface.  Regarding its middle cohomology as a
lattice with respect to the intersection form, we can write \cite[\S 1]{LP}
\begin{equation}
\label{eq: K3 lattice}
H^2(X,\bZ)\simeq U^3 \oplus E_8(-1)^2
\end{equation}
where 
\[
U = \langle e, f\rangle, \quad\textup{with intersections}\quad
\begin{array}{r|cc}
	& e   & f  \\
\hline
e       & 0   & 1  \\
f       & 1   & 0  
\end{array}
\]
and $E_8$ is the positive definite lattice arising from the corresponding root system,
i.e., the unique positive definite even unimodular lattice of rank eight.  
Let $e$ and $f$ denote the generators of the first summand $U$ in~\eqref{eq: K3 lattice}, and
$h\in H^2(X,\bZ)$ a primitive
vector with $h\cdot h=2d>0$.  The isomorphism~\eqref{eq: K3 lattice} can be chosen so that
\[
h\mapsto e+df.
\]
Writing $v=e-df$, we have
\[
h^{\perp}\simeq \bZ v \oplus  \Lambda' ,\qquad \textup{where }  \Lambda' := U^2 \oplus E_8(-1)^2.
\]

Let $(X,h)$ be a polarized K3 surface of degree $2d$, i.e., $h\cdot h=2d$;
assume that $\Pic(X)$ is generated by $h$.  Using the exponential sequence, two-torsion
elements of the Brauer group of $X$ may be interpreted as elements
$$\alpha \in H^2(X,\bZ/2\bZ)/\left<h\right>.$$
Under this identification, we can express
\[
\alpha=n \bar{f}+{\bar\lambda_{\alpha}}, \quad n=0,1,
\]
where $\bar{f}$ is the image of $f$ and ${\bar\lambda_{\alpha}}$ is the 
image of some  $\lambda_{\alpha} \in \Lambda'$.

Choose $\mu \in \Lambda'$ satisfying 
\[
\mu\cdot \lambda_{\alpha}\equiv 1\bmod{2}.
\]  
Using the non-degenerate cup product on $H^2(X,\bZ)$, consider
\[
\alpha^{\perp} \subset h^{\perp} \subset H^2(X,\bZ),
\]
where the first subgroup has index two when $\alpha \neq 0$.  
If $n=0$ then 
$$\alpha^{\perp}=\bZ v\oplus\{\lambda'\in \Lambda': \lambda'\cdot \lambda_{\alpha} \equiv 0\bmod{2} \},
$$
a lattice with discriminant group 
$(\bZ/2d\bZ)\oplus (\bZ/2\bZ)^2$, where the last two summands are generated
by $\lambda_{\alpha}/2$ and $\mu$.
If $n=1$ then
$$\alpha^{\perp}=\bZ (v+\mu)+ \{\lambda'\in\Lambda':\lambda'\cdot \lambda_{\alpha}
					\equiv 0\bmod{2} \},$$
a lattice with discriminant group generated by 
$(-v+2d\lambda_{\alpha})/4d,$
which is therefore $\bZ/8d\bZ$.  

General results on quadratic forms (see, for example, \cite{Nikulin})
make it possible to classify indefinite quadratic forms with prescribed
discriminant group $H$, provided the rank of the form is significantly
larger than the number of generators of $H$.
In particular, van Geemen \cite[Proposition~9.2]{vanGeemen}
classifies isomorphism classes of lattices $\alpha^\perp$ arising from this construction:
\begin{itemize}
\item{if $n=0$ there is a unique such lattice, up to isomorphism;}
\item{if $n=1$ and $d$ is even then there is a unique such lattice up
to isomorphism;}
\item{if $n=1$ and $d$ is odd then there are two such lattices up to
isomorphism, depending on the parity of 
$\lambda_{\alpha}\cdot \lambda_{\alpha}/2$.}
\end{itemize}

He goes further when $d=1$, offering geometric constructions of varieties
having primitive Hodge structure isomorphic to $\alpha^{\perp}$.  We
elaborate on his description: \\

\noindent {\bf Case $n=0$:}  Let $W\subset \bP^2 \times \bP^2$ denote a smooth 
hypersurface of bidegree $(2,2)$ and $Y\ra \bP^2 \times \bP^2$ the double
cover branched along $W$.  Let $h_1$ and $h_2$ denote the divisors on $Y$
obtained by pulling back the hyperplane classes from the factors.  We have
intersections:
$$\begin{array}{r|ccc}
        & h_1^2 & h_1h_2 & h_2^2 \\
\hline 
h_1^2  &  0     &   0    &    2  \\
h_1 h_2 & 0     &   2    &    0  \\
h_2^2   & 2     &   0    &    0
\end{array}
$$
The non-zero Hodge numbers of $Y$ are:
$$h^{00}=h^{44}=1, h^{11}=h^{33}=2, h^{13}=h^{31}=1, h^{22}=22.$$
Consider the weight-two Hodge structure
$$  \left<h_1^2,h_1h_2,h_2^2\right>^{\perp} \subset H^4(Y)(1)$$
having underlying lattice $M$, with respect to the intersection form.
A computation of the discriminant group of $M$ implies that
$$M\simeq \alpha^{\perp}.$$ 

\medskip

\noindent {\bf Case $n=1, \lambda_{\alpha}\cdot \lambda_{\alpha} \equiv 0\pmod{4}$:}
In this case, there exists a {\em primitive} embedding
$$\alpha^{\perp} \hookrightarrow U^3 \oplus E_8(-1)^2,$$
unique up to automorphisms of the source and target.  We can interpret
the image as the primitive cohomology of a polarized K3 surface $(S,f)$
with $f\cdot f=8$.  \\

\noindent {\bf Case $n=1, \lambda_{\alpha}\cdot \lambda_{\alpha} \equiv 2\pmod{4}$:}
Let $Y$ be a cubic fourfold containing a plane $P$, with hyperplane class $h$.  
We have the intersections:
$$\begin{array}{r|cc}
	& h^2 & P \\
\hline
h^2  & 3     & 1  \\
P    & 1     & 3 
\end{array}$$
The non-zero Hodge numbers of $Y$ are
$$h^{00}=h^{44}=1, h^{11}=h^{33}=1, h^{13}=h^{31}=1, h^{22}=21.$$
The weight-two Hodge structure
$$\left< h^2,P \right>^{\perp} \subset H^4(Y)(1)$$
has underlying lattice isomorphic to $\alpha^{\perp}$.  

\

The last two geometric constructions yield explicit unramified Azumaya algebras
over the degree two K3 surface.  The connection between cubic fourfolds containing
planes and quaternion algebras over the K3 surface can be found in \cite{HVAV};
the other construction goes back to Mukai \cite{Mukai}:
A degree eight K3 surface $S$ is generally a complete
intersection of three quadrics in $\bP^5$, and the discriminant curve of the corresponding
net is a smooth plane sextic.  Let $X$ be a degree-two K3 surface obtained as the double
cover of $\bP^2$ branched along this sextic.  The variety $\calF$ parametrizing
maximal isotropic subspaces of the
quadrics cutting out $S$ admits a morphism 
(cf. \cite[\S 3]{HVAV}) $\calF \ra X,$
which is smooth with geometric fibers isomorphic to $\bP^3$. 

In this paper, we focus on the first case, and use the resulting Azumaya algebra for arithmetic purposes.

%****************************************************************************
\section{Unramified conic bundles}
\label{s: unramified}

Let $k$ be an algebraically closed field of characteristic $\neq 2$, and let $W$ be a \defi{type $(2,2)$ divisor} on $\bP^2\times\bP^2$, that is, hypersurface of bidegree $(2,2)$.
%\[
%\bP^2\times\bP^2 = \Proj k[x_0,x_1,x_2] \times \Proj k[y_0,y_1,y_2].
%\]
The two projections $\pi_1\colon W \to \bP^2$ and $\pi_2\colon W \to \bP^2$ define  conic bundle structures on $W$ that are ramified, respectively, over plane sextic curves $C_1$ and $C_2$. Assume that $W$ has at worst isolated singularities. Let $Y \to \bP^2 \times \bP^2$ be the double cover branched along $W$. Composing this map with the projections onto the factors we obtain two quadric surface bundles $q_i\colon Y \to \bP^2$, also ramified over the curves $C_i$, for $i = 1, 2$, respectively. 

Let $\phi_i\colon X_i \to \bP^2$ be a double cover of $\bP^2$ ramified over $C_i$. 
If $C_i$ is smooth then $X_i$ is a K3 surface of degree $2$. 

\begin{lemma}
\label{lem:W smooth}
If $Y$ (equivalently, $W$) is not smooth then neither $C_1$ nor $C_2$ is smooth. On the other hand, if $Y$ is smooth, then the curve $C_i$ is singular if $q_i$ has a geometric fiber of rank $1$, $i = 1, 2$.
\end{lemma}

\begin{proof}
%Interpret $C_i$ as the discriminant divisor of the conic fibration
%$\pi_i:W \ra \bP^2$. 
%Over fields of characteristic zero, the multiplicity of the discriminant 
%at a point can be expressed in terms of the Milnor numbers of the singularities of the
%fiber and the total space over that point 
%\cite[2.8.3]{Te}.  However, we are not aware of a reference
%for this over fields of positive characteristic, so we prove the special case
%relevant to our application.  
An easy application of the Jacobian criterion shows that $Y$ is smooth if and only if $W$ is smooth. We use the later scheme to prove the remaining claims of the lemma.

Let $w \in W$ be a singular point, $c_i\in C_i$ its images under projection,
$(u_0,u_1)$ local coordinates of the first $\bP^2$ centered at $c_1$,
and $(v_0,v_1)$ local coordinates of the second $\bP^2$ centered at $c_2$.  
The defining equation of $W$ takes the form
\[
a(u_0,u_1)v_0^2+b(u_0,u_1)v_0v_1+2d(u_0,u_1)v_1^2+c(u_0,u_1)v_0+
e(u_0,u_1)v_1=0,
\]
where the coefficients are quadratic in $u_0$ and $u_1$, and 
$c(0,0)=e(0,0)=0$.  The defining equation for $C_1$ is therefore
\[
\det \left( \begin{matrix} 2a & b & c \\
			      b & 2d & e \\
			      c & e &  0 \end{matrix} \right)=0.
\]
Expanding this out, we get
$$bce-ae^2-c^2d=0,$$
where each term vanishes to order $\ge 2$ at $c_1=(0,0)$.  

The last statement of the lemma is a consequence of \cite[Prop.~1.2]{Beau}.
\end{proof}

\begin{theorem}
\label{thm: Stein}
Let $\calO$ denote a discrete valuation ring with residue field $\bF$,
of characteristic~$\neq~2$.  
Let $\calW$ be a type $(2,2)$ divisor in $\bP^2\times \bP^2$ flat over $\Spec\calO$,
and $\calY \ra \bP^2 \times \bP^2$ a
double cover simply branched along $\calW$.  Let $q_i\colon\calY \ra \bP^2$ denote the quadric
surface bundle obtained by projecting onto the $i$-th factor, and let
$\calC_i\subset \bP^2$ be its discriminant divisor.
Assume that $\calC_i$ is flat over $\calO$, and that $(\calC_i)_{\bF}$ is smooth for some $i \in \{1,2\}$. 

Let $r_i\colon \calF_i \to \bP^2$ be the relative variety of lines of $q_i$.  Then the Stein factorization
\[
r_i\colon \calF_i \to \calX_i \xrightarrow{\phi_i} \bP^2
\]
consists of a smooth $\bP^1$-bundle followed by a degree-two cover of $\bP^2$, which is a K3 surface.
\end{theorem}

\begin{proof}
By Lemma~\ref{lem:W smooth}, smoothness of $(\calC_i)_{\bF}$ implies smoothness of 
${\calW}_{\bF}$, and hence of ${\calY}_{\bF}$. The same lemma shows that the fibers 
of $(q_i)_\bF$ have at worst isolated singularities. On the other hand, the morphism $q_i\colon\calY \to \bP^2$ is flat, and thus $\calY$ is a regular scheme. Geometric fibers of $q_i$ over the generic point of $\Spec\calO$ with non-isolated singularities specialize to geometric fibers over the closed point with non-isolated singularities. Hence the fibers of $q_i$ have isolated singularities. The theorem now follows directly from \cite[Proposition 3.3]{HVAV}:  The Stein factorization of the variety of maximal isotropic
subspaces of a family of even-dimensional quadric hypersurfaces with (at worst)
isolated singularities is isomorphic to the discriminant double cover of the base. 

Since the morphism $\calC_i \to \bP^2$ is flat, smoothness of $(\calC_i)_\bF$ implies that $\calC_i$ is regular. Hence $\calX_i$ is a K3 surface over $\Spec\calO$.
\end{proof}

The smooth $\bP^1$-bundle $r_i\colon \calF_i \to \calX_i$ may be interpreted as a two-torsion element of $\Br(\calX_i)$. Without loss of generality, assume in the hypotheses of Theorem~\ref{thm: Stein} that $(\calC_1)_\bF$ is smooth; we omit the subscript $i = 1$ from here on in. We give an explicit quaternion algebra over $\kk(\calX)$ representing the Brauer class of $\calF \to \calX$.   Express
\[
\bP^2 \times \bP^2 = \Proj \calO[x_0,x_1,x_2] \times_{\calO} \Proj \calO[y_0,y_1,y_2]
\]
so the equation for $\calW$ takes the form
\begin{equation}
\label{eq: calW}
\begin{split}
0 & = A(x_0,x_1,x_2)y_0^2 + B(x_0,x_1,x_2)y_0y_1 + C(x_0,x_1,x_2)y_0y_2 \\
  &\quad + D(x_0,x_1,x_2)y_1^2 + E(x_0,x_1,x_2)y_1y_2 + F(x_0,x_1,x_2)y_2^2,
\end{split}
\end{equation}
for some homogeneous quadratic polynomials $A,\dots,F \in \calO[x_0,x_1,x_2]$. 
The coefficients are unique modulo multiplication by a unit in $\calO$.  

Consider the bigraded ring $\calO[x_0,x_1,x_2,y_0,y_1,y_2,v]$ where
$$\deg(x_i)=(1,0),\quad \deg(y_i)=(0,1), \quad \deg(v)=(1,1),$$
and let 
$$R:=\bigoplus_{n\in \bZ} \calO[x_0,x_1,x_2,y_0,y_1,y_2,v]_{(n,n)}$$ 
denote the graded subring generated by
elements of bidegree $(n,n)$ for some $n$.  
Then an equation for $\calY \subset \Proj R$
is 
\begin{equation}
\label{eq: calY}
\begin{split}
v^2 &= A(x_0,x_1,x_2)y_0^2 + B(x_0,x_1,x_2)y_0y_1 + C(x_0,x_1,x_2)y_0y_2 \\
  &\quad + D(x_0,x_1,x_2)y_1^2 + E(x_0,x_1,x_2)y_1y_2 + F(x_0,x_1,x_2)y_2^2,
\end{split}
\end{equation}
The quadric surface bundle $q\colon \calY \to \bP^2$ is ramified over the curve
\begin{equation}
\label{eq: C1}
\calC : \quad
\det\begin{pmatrix}
2A & B & C & 0 \\
B & 2D & E & 0 \\
C & E & 2F & 0 \\
0 & 0 & 0 & -2
\end{pmatrix} = 0.
\end{equation}
Thus, after rescaling, the K3 surface $\calX$ is described by the hypersurface
\begin{equation}
\label{eq: K3}
w^2 = -\frac{1}{2}\cdot
\det(M),
\end{equation}
in $\bP(1,1,1,3)$, where $M \in \Mat_3(\calO[x_0,x_1,x_2])$ is the leading $3\times 3$ principal minor of the matrix in~\eqref{eq: C1}.

The discussion in \cite[\S 3.3]{HVAV} shows that the generic fiber of the map $\calF \to \calX$ is the Severi-Brauer conic in $\Proj\kk(\calX)[Y_0,Y_2,Y_2]$ given by
\begin{equation}
\label{eq:conic bundle}
AY_0^2 + BY_0Y_1 + CY_0Y_2 + DY_1^2 + EY_1Y_2 + FY_2^2 = 0.
\end{equation}
Essentially, given a smooth quadric surface whose discriminant
double cover is split, each component of the variety of lines on the surface
is isomorphic to a smooth hyperplane section of the surface.  
Let
\[
M_A := 4DF - E^2,\quad 
M_D := 4AF - C^2,\quad\text{and}\quad
M_F := 4AD - B^2.
\]
Completing squares in~\eqref{eq:conic bundle}, and renormalizing, we obtain
\[
Y_0^2 = -\frac{M_F}{4A^2}Y_1^2 - \frac{\det(M)}{2A\cdot M_F}Y_2^2.
\]
Hence, by~\cite[Corollary~5.4.8]{GilleSzamuely}, the conic \eqref{eq:conic bundle} corresponds to the Hilbert symbol
\begin{equation}
\label{eq:quaternion algebra}
\left(-\frac{M_F}{4A^2},-\frac{\det(M)}{2A\cdot M_F}\right).
\end{equation}
Write $\calA$ for the class of this symbol in $\Br(\kk(\calX))$; $\calA$ is unaffected by multiplication by squares in either entry of a representative symbol. Since $-\frac{1}{2}\det(M)$ is a square in $\kk(\calX)^\times$, we see that 
\[
(-M_F,A\cdot M_F) = (-M_F,A)
\]
is another representative of $\calA$ (the equality uses the multiplicativity of the Hilbert symbol and the relation $(-M_F,M_F) = 1$~\cite[III, Proposition~2]{Serre}). Here we have the usual abuse of notation: the entries are not rational functions, though they are homogeneous polynomials of even degree. 

Depending on how we complete squares and renormalize~\eqref{eq:conic bundle}, we may obtain several representatives of $\calA$:
\begin{equation}
\label{eq:representatives}
\begin{split}
(-M_F,A),\quad &(-M_D,A),\quad (-M_F,D),\\
(-M_A,D),\quad &(-M_D,F),\quad (-M_A,F).
\end{split}
\end{equation}

\begin{proposition}
\label{prop: brauer elt}
Let $X$ be a K3 surface of degree $2$ over a number field $k$, given as a sextic in the weighted projective space $\bP(1,1,1,3) = \Proj k[x_0,x_1,x_2,w]$ of the form
\begin{equation*}
%\label{eq: explicit K3}
w^2 = -\frac{1}{2}\cdot
\begin{pmatrix}
2A & B & C \\
B & 2D & E \\
C & E & 2F
\end{pmatrix},
\end{equation*}
where $A,\dots,F \in k[x_0,x_1,x_2]$ are homogenous quadratic polynomials. Then the class $\calA$ of the quaternion algebra $(B^2 - 4AD,A)$ in $\Br(\kk(X))$ extends to an element of $\Br(X)$.
\end{proposition}

\begin{proof}
Let $\calO$ be the valuation ring at some finite place of $k$ where $X$ has good reduction. The proposition follows directly from Theorem~\ref{thm: Stein} and the subsequent discussion, keeping track of what is happening over the generic point of $\Spec\calO$. Indeed, define $\calW$ and $\calY$, respectively, by~\eqref{eq: calW} and~\eqref{eq: calY}. The resulting curve $(\calC_1)_k$ is the branch curve of the double cover $X \to \bP^2$, which is smooth because $X$ is a K3 surface, by hypothesis.
\end{proof}

\begin{remark}
The assortment of quaternion algebras~\eqref{eq:representatives} representing the class $\calA$ of Proposition~\ref{prop: brauer elt} is useful for the computation of the invariant map on the image of the evaluation map $\textrm{ev}_\calA\colon X(\bA) \to \bigoplus_v \Br(k_v), (P_v) \mapsto \big(\calA(P_v)\big)$. The industrious reader can check that at every local point of $X$, either the first, fourth of fifth representatives in $\calA$ in our list is well-defined; we shall not use this observation directly.
\end{remark}

%****************************************************************************
\section{Local Invariants}
\label{s: local invariants}

Let $X$ be a smooth projective geometrically integral variety over a number field $k$. For $\calS \subseteq \Br(X)$, let
\[
X(\Adeles)^\calS := \left\{ (P_v) \in X(\Adeles) : \sum_v \inv_v \calA(P_v) = 0 \textup{ for all }\calA \in \calS \right\}.
\]
The inclusion $X(k) \subseteq X(\Adeles)^\calS$ follows from class field theory.  See~\cite[\S5.2]{SkorobogatovBook} for details. The local invariants $\inv_v\calA(P_v)$ can be nonzero only at a finite number of places: the archimedean places of $k$, the places of bad reduction of $X$, and places where the class $\calA$ is ramified.

We begin this section by explaining how recent work of Colliot-Th\'el\`ene and Skorobogatov~\cite{CTS-TAMS} shows that local invariants are \emph{constant} at certain finite places $v$ of bad reduction for $X$ where the singular locus satisfies a technical hypothesis. Specializing to the case where $X$ is a K3 surface over a number field $k$ as in Proposition~\ref{prop: brauer elt}, this technical hypothesis is satisfied provided the singular locus at $v$ consists of $r < 8$ ordinary double points (Lemma~\ref{lem:no cyclic covers}). 

We then show that the class $\calA$ of Proposition~\ref{prop: brauer elt} can ramify only over infinite places, $2$-adic places, and places of bad reduction for $X$. Finally, in the special case $k = \bQ$, we give sufficient conditions for local invariants of $\calA$ to be trivial at $2$-adic points and nontrivial at real points.

%%%%%%%%%%%%%%%%%%%%%%
\subsection{Places of bad reduction with mild singularities}

In this section we use the following notation: $k$ is a finite extension of $\bQ_p$ with a fixed algebraic closure $\kbar$, $\calO$ denotes the ring of integers of $k$, and $\bF$ denotes its residue field. We let $X$ be a smooth, proper, geometrically integral variety over $k$ and write $\pi\colon\calX \to \Spec\calO$ for a flat proper morphism with $X = \calX\times_\calO k$.

The following proposition is a straightforward refinement of~\cite[Proposition~2.4]{CTS-TAMS}, using ideas in the remark on the case of bad reduction in~\cite[\S2]{CTS-TAMS}. We include the details here for the reader's convenience.

\begin{proposition}%[Colliot-Th\'el\`ene, Skorobogatov]
\label{prop:constant eval map}
Let $\ell\neq p$ be a prime.  Assume that $\calX$ is regular with geometrically integral fibers over $\Spec\calO$, and that the smooth locus $\calX_\bF^{\sm}$ of the closed fiber is irreducible and has no connected unramified cyclic geometric coverings of degree $\ell$. If $X(k) \neq \emptyset$, then, for $\calA \in \Br(X)\{\ell\}$, the image of the evaluation map $\ev_\calA\colon X(k) \to \Br(k)$ consists of one element.
\end{proposition}

\begin{proof}
Let $\calZ$ be the largest open subscheme of $\calX$ that is smooth over $\Spec\calO$; note that $\calZ\times_\calO k = X$. Write $\calZ_\bF$ for its closed fiber, and note that $\calZ_\bF = \calX_\bF^{\sm})$. Let $\calZ_\bF^{(1)}$ denote the set of closed integral subvarieties of $\calZ_\bF$ of codimension $1$. In~\cite[Prop.\ 1.7]{Kato}, Kato shows there is a complex
\[
\Br(X)[\ell^n] \xrightarrow{\res} H^1\big(\kk(\calZ_\bF),\bZ/\ell^n\bZ\big) \to \bigoplus_{Y\subset \calZ_\bF^{(1)}} H^0\big(\kk(Y),\bZ/\ell^n\bZ(-1)\big).
\]
(In Kato's notation, take $q = -1$, $i = -2$, $n \mapsto \ell^n$, and $X \mapsto \calZ$.) We claim that for $\calA \in \Br(X)[\ell^n]$, the residue $\res(\calA) \in H^1\big(\kk(\calZ_\bF),\bZ/\ell^n\bZ\big)$ lies in the subgroup $H^1\big(\calZ_\bF,\bZ/\ell^n\bZ\big)$. Indeed, the group $H^1\big(\kk(\calZ_\bF),\bZ/\ell^n\bZ\big)$ classifies  connected cyclic covers of $\calZ_\bF$. By Kato's complex, the cover $\calW \to \calZ_\bF$ corresponding to $\res(\calA)$ is unramified in codimension one, and hence, by the Zariski-Nagata purity theorem~\cite[Expos\'e X, Th\'eor\`eme 3.1]{SGAI}, $\calW \to \calZ_\bF$ is unramified.

The long exact sequence of low degree terms associated to the spectral sequence
\[
E_2^{p,q} := H^p\big(\bF,H^q_{\et}\big(\calZ_{\overline{\bF}},\bZ/\ell^n\bZ\big)\big) \implies
H^{p+q}_{\et}(\calZ_\bF,\bZ/\ell^n\bZ)
\]
starts as follows:
\[
0 \to H^1\big(\bF,\bZ/\ell^n\bZ\big) \to H^1\big(\calZ_\bF,\bZ/\ell^n\bZ\big)
\to H^1\big(\calZ_{\overline{\bF}},\bZ/\ell^n\bZ\big)
\]
Since, by hypothesis, $\calZ_\bF$ has no connected unramified cyclic geometric coverings of degree $\ell$, we have $H^1\big(\calZ_{\overline{\bF}},\bZ/\ell^n\bZ\big) = 0$.

Local class field theory shows that the residue map $\Br(k)[\ell^n] \to H^1\big(\bF,\bZ/\ell^n\bZ\big)$ is an isomorphism.  We conclude that for any $\calA \in \Br(X)\{\ell\}$, there exists $\alpha \in \Br(k)\{\ell\}$ such that $\calA - \alpha$ has trivial residues along any codimension one subvariety of $\calX$. By Gabber's purity theorem~\cite{Fujiwara}, it follows that $\calA - \alpha \in \Br(\calX)\{\ell\} \subseteq \Br(X)\{\ell\}$.

A valuation argument shows that $X(k) = \calX(\calO) = \calZ(\calO)$; see~\cite[proof of Lemma~1.1(b)]{SkorobogatovAJM}. Since $\Br(\calO) = 0$, we conclude that the images of the evaluation maps $\ev_\calA$ and $\ev_\alpha$ in $\Br(k)$ coincide; the latter consists of one element.
\end{proof}

\begin{lemma}
\label{lem:no cyclic covers}
Suppose that $p\neq 2$. Let $X$ be a K3 surface defined over $k$, and let $\pi\colon\calX \to \Spec\calO$ be a flat proper morphism from a regular scheme with $X = \calX\times_\calO k$. Assume that the singular locus of the closed fiber $\calX_0:=\calX_{\overline{\bF}}$ has $r<8$ points, each of which is an ordinary double point. Then the smooth locus $U\subset \calX_0$ has no connected unramified cyclic covers of prime degree $\ell \neq p$.
\end{lemma}

\begin{proof}
Consider an algebraically closed field $F$ of characteristic
different from $\ell$.  Let $Y$ be a separated integral scheme over $F$ with $\Gamma(Y,{\sO}^*_Y)=F^*$;
this is the case if $Y$ is proper, or a dense open subset of a proper scheme with
complement of codimension $\ge 2$.  
Then degree $\ell$ cyclic \'etale covers of $Y$ are classified by $H^1_\et(Y,\mu_{\ell})$
\cite[ch.III]{Milne}.  
The Kummer exact sequence \cite[p.125]{Milne} implies that $H^1_\et(Y,\mu_{\ell})=\Pic(Y)[\ell]$,
the $\ell$-torsion subgroup.  

Combining the canonical homomorphism from the
Picard group to the Weil class group and the restriction homomorphism 
on class groups yields
$$\Pic(\calX_0) \subset \Cl(\calX_0) \simeq \Cl(U) \simeq  \Pic(U).$$
The quotient $\Pic(U)/\Pic(\calX_0)$ is two-torsion.  Indeed, ordinary double points
are \'etale locally isomorphic to quadric cones, whose local class group 
equals $\bZ/2\bZ$ (generated by the ruling).   Thus for each closed point $x\in \calX_0$,
the quotient $\Cl(\Spec \sO_{\calX_0}\setminus \{x\})/\Cl(\Spec \sO_{\calX_0})$ is annihilated by two
\cite[\S 14]{Lipman}.  

If $\ell \neq 2$ then this computation shows that $\Pic(U)[\ell]=\Pic(\calX_0)[\ell]$, whence
degree $\ell$ cyclic \'etale covers of $U$ extend to $\calX_0$.  Consider, the specialization
homomorphism \cite[\S 20.3]{Fulton}
$$\Pic(X_{\overline{k}}) \ra \Pic(\calX_0).$$
We claim this is injective and the cokernel has torsion annihilated by $p$;
this implies that $\Pic(\calX_0)[\ell]=0$.  

To prove the claim, replace $k$ by the ramified quadratic extension $k'$ with ring of integers $\calO'$,
so that 
$\calX'=\calX \times_{\calO} \calO'$ is singular over the double points of $\calX_0$.  
Concretely, given $\pp$ a uniformizer of $\calO$, $\pp'=\sqrt{\pp}$
the corresponding uniformizer of $\calO'$, and $x\in \calX_0$ an ordinary double point,
then the \'etale local equation of $\calX$
$$\pp=uv+w^2$$
pulls back to
$$(\pp')^2=uv+w^2.$$
Let $\tilde{\calX} \ra \calX'$ denote the blow-up of the resulting singularities,
with central fiber the union of the proper transform of $\calX_0$ and the exceptional divisors
$$\tilde{\calX}_0=S \cup E_1 \cup \cdots \cup E_r, \quad \overline{E_j} \simeq \bP^1 \times \bP^1.$$
(At the cost of passing to an algebraic space, we could blow down $E_1,\ldots,E_r$ along one of the rulings of $\bP^1 \times \bP^1$.)  
Note that $S$ is a K3 surface and the specialization
$$\Pic(X_{\overline{k}}) \ra \Pic(S)$$
is injective with cokernel having torsion annihilated by $p$ \cite[Prop.~3.6]{MP}.  
However, this admits a factorization
$$\Pic(X_{\overline{k}}) \ra \Pic(\calX_0) \ra \Pic(S)$$
where the second arrow is injective.  Thus the cokernel of the first arrow
has torsion annihilated by $p$.  

We now focus on the case $\ell=2$.  We continue to use $\beta:S\ra \calX_0$
to denote the minimal resolution of $\calX_0$;  let $F_1,\ldots,F_r$ denote
the exceptional divisors of $\beta$, which satisfy
$$F_1^2=\cdots=F_r^2=-2, \quad F_iF_j=0, i\neq j,$$
because $\calX_0$ has ordinary double points.  Note that $S$ is still a K3
surface.  

There may exist \'etale double covers of $U$ that fail to extend to
\'etale covers of $\calX_0$.  Given an \'etale double cover $V\ra U$,
let $\varpi:T\ra S$ denote the normalization of $S$ in the function field of
$V$.  Since $T$ is normal, $\varpi$ is a flat morphism \cite[Ex.~18.17]{Eis},
\'etale away from $F_1\cup \cdots \cup F_r$.  Moreover, by purity
of the branch locus, $\varpi$ is
branched over some subset 
$$\{F_{j_1},\ldots,F_{j_s} \} \subset \{ F_1,\ldots,F_r\}.$$
Since the characteristic is odd, $\varpi$ is simply
branched over these curves.  Consequently, 
$\sum_{i=1}^s F_{j_i}=2D$ for some $D\in \Pic(S)$, 
hence $s\equiv 0\pmod{4}$, i.e., $s=0$ or $4$.  The case 
$s=0$ is impossible, since this would mean that $S$ admits 
an \'etale cyclic cover with degree prime to the characteristic.  
The case $s=4$ is also impossible:  We have
$$\chi(\sO_S(-D))=1$$
but 
$$h^2(\sO_S(-D))=h^0(\sO_S(D))=
0,$$
as any effective divisor supported in the $F_j$ (like $2D$) is rigid.  
On the other hand, since a divisor and its negative cannot
both be effective, we find
$$h^0(\sO_S(-2D))=0 \text{ which implies }
h^0(\sO_S(-D))=0.$$
Therefore 
$h^1(\sO_S(-D))=-1$,
which is a contradiction.
\end{proof}

\begin{remark}
When $r = 8$, it is possible that a smooth resolution of $\calX_0$ is a K3 surface with a Nikulin involution, in which case the smooth locus $U \subset \calX_0$ has a connected unramified cyclic double cover~\cite{SvG}.
\end{remark}

%%%%%%%%%%%%%%%%%%%%%
\subsection{Places where $\calA$ can ramify}

\begin{lemma}
\label{lem:most invariants}
Let $X$ be a K3 surface over a number field $k$ as in Proposition~\ref{prop: brauer elt}. Let $v$ be a finite place of good reduction for $X$, and assume that $v$ is not $2$-adic.
Then $\calA$ does not ramify at $v$. Consequently, $\inv_v \calA(P) = 0$ for all $P \in X(k_v)$.
\end{lemma}

\begin{proof}
We may assume without loss of generality that the coefficients of $A,\dots,F$ are integral. Let $\calO_v$ be the ring of integers of $k_v$, and $\bF_v$ its residue field. Since $X$ is smooth proper over $k$ and has good reduction at $v$, there is a smooth proper morphism $\calX \to \Spec\calO_v$ with $X_{k_v} = \calX\times_{\calO_v} k_v$. 
%In fact, $\calX$ is cut out by the same equation as $X$, considered in $\Proj \calO_v[x_0,x_1,x_2,w]$. 
We will show that the class $\calA\otimes k_v$ can be spread out to a class in $\Br(\calX)$. Since, by the valuative criterion of properness, we have $\calX(\calO_v) = X(k_v)$, it will follow that $\calA(P) \in \Br(\calO_v) = 0$ for every point $P \in X(k_v)$, establishing all the claims of the proposition.

Define $\calW$ and $\calY$ over $\calO_v$, respectively, by~\eqref{eq: calW} and~\eqref{eq: calY}. The quadric surface bundle $(q_1)_\bF\colon \calY_\bF \to \bP^2_\bF$ ramifies over the discriminant curve of $\calX_{\bF_v} \to \bP_{\bF_v}^2$, which is smooth, because $X$ has good reduction at $v$. By Theorem~\ref{thm: Stein}, there exists a smooth $\bP^1$ bundle $\calF \to \calX$, whose corresponding two-torsion class in $\Br(\calX)$ is represented by the quaternion algebra $(B^2 - 4AB,A)$, by the discussion following Theorem~\ref{thm: Stein}. Thus $\calA \in \Br(\calX)$, as claimed.
\end{proof}

%%%%%%%%%%%%%%%%%%%%%%%
\subsection{Real and $2$-adic invariants}

In this section we use the notation of Proposition~\ref{prop: brauer elt}, specializing to the case $k = \bQ$. The following lemma gives a sufficient condition to guarantee that the local invariants of $\calA$ at real points of $X$ are always non-trivial. 

\begin{lemma}
\label{lem:real place}
Suppose that the quadratic forms $A$, $B$, $C$, $D$, $E$ and $F$ satisfy
\begin{enumerate}
\item $A$, $D$ and $F$ are negative definite,
\item $B$, $C$ and $E$ are positive definite
\end{enumerate}
Then, for any real point of $X$, we have
\[
M_A > 0,\quad
M_D > 0\quad\text{and}\quad
M_F > 0.
\]
\end{lemma}

\begin{proof}
First, observe that we can write $\frac{1}{2}\det(M)$ as
\begin{equation}
\label{eq:det M_A}
A\cdot M_A - (C^2D + B^2F - BCE).
\end{equation}
Let $P$ be a real point of $X$, so that $\frac{1}{2}\det(M) \leq 0$ holds at $P$. Our hypotheses on $A,\dots,F$ imply that
\[
(C^2D + B^2F - BCE)(P) < 0.
\]
Suppose first that $M_A \leq 0$. Then at $P$ we have
\[
\frac{1}{2}\det(M) = \underbrace{A}_{<0}\cdot \underbrace{M_A}_{\leq 0} - (\underbrace{C^2D + B^2F - BCE}_{< 0}) > 0,
\]
a contradiction. Hence $M_A > 0$ at $P$. A similar argument shows the remaining two cases.
\end{proof}

\begin{corollary}
\label{cor:real place}
Suppose the hypotheses of Lemma~\ref{lem:real place} hold. Then the local invariant of $\calA$ at every real point of $X$ is nontrivial.
\end{corollary}

\begin{proof}
It suffices to show that, for any real point $P$ of $X$, there is a quaternion algebra representing $\calA$ whose entries are both negative at $P$. Using the six representatives~\eqref{eq:representatives} of $\calA$, together with Lemma~\ref{lem:real place}, the result follows.
\end{proof}

Next, we write down a sufficient condition to guarantee that the local invariant map on $\calA$ is constant and trivial on $2$-adic points.  Write $v_2\colon\bQ_2 \to \bZ\cup \{\infty\}$ for the standard $2$-adic valuation.  Recall that $a \in \bQ_2^\times$ is a square if and only if $v_2(a)$ is even and if $a/2^{v_2(a)} \equiv 1 \bmod 8$.

Let $P = [x_0:x_1:x_2:w]$ denote a $2$-adic point of $X$. We may assume without loss of generality that $x_0$, $x_1$ or $x_2$ are elements of $\bZ_2$, at least one of which is a $2$-adic unit. Suppose first that $x_0$ is a $2$-adic unit, so that $v_2(x_0) = 0$. We use  the representative $(-M_F,A)$ of $\calA$ to evaluate invariants at $P$. Write
\[
A = A_1x_0^2 + A_2x_0x_1 + A_3x_0x_2 + A_4x_1^2 + A_5x_1x_2 + A_6x_2^2,
\]
and suppose that the coefficients of $A$ satisfy
\[
A_1 \equiv 1 \bmod 8, \quad\text{and}\quad
v_2(A_i) \geq 3 \text{ for } i = 2,\dots,6.
\]
Then, at $P$, we have $A \equiv 1 \bmod 8$ (since $v_2(x_0) = 0$) so $A$ is a $2$-adic square. It follows that $\inv_2\calA(P) = 0$, provided that $M_F(P) \neq 0$. To ensure this, we impose restrictions on the coefficients of the quadratic form
\[
B = B_1x_0^2 + B_2x_0x_1 + B_3x_0x_2 + B_4x_1^4 + B_5x_1x_2 + B_6x_2^2.
\]
Suppose that
\[
v_2(B_1) = 0,\quad\text{and}\quad
v_2(B_i) \geq 1 \text{ for } i = 2,\dots,6.
\]
Then, since $v_2(x_0) = 0$, it follows that 
\[
v_2(M_F(P)) = v_2(B(P)) = 0
\]
and hence $M_F\neq 0$ at $P$.

To ensure that $2$-adic invariants of $\calA$ are trivial at points where $v_2(x_1) = 0$, we use the representative $(-M_A,D)$ of $\calA$ and constrain the coefficients of $D$ and $E$, respectively, in a manner analogous to how we constrained the coefficients of $A$ and $B$. We proceed similarly for $2$-adic points with $v_2(x_2) = 0$.  We summarize our discussion in the following lemma.

\begin{lemma}
\label{lem:2-adic invariants}
Write
\begin{align*}
A &= A_1x_0^2 + A_2x_0x_1 + A_3x_0x_2 + A_4x_1^2 + A_5x_1x_2 + A_6x_2^2, \\
B &= B_1x_0^2 + B_2x_0x_1 + B_3x_0x_2 + B_4x_1^2 + B_5x_1x_2 + B_6x_2^2, \\
C &= C_1x_0^2 + C_2x_0x_1 + C_3x_0x_2 + C_4x_1^2 + C_5x_1x_2 + C_6x_2^2, \\
D &= D_1x_0^2 + D_2x_0x_1 + D_3x_0x_2 + D_4x_1^2 + D_5x_1x_2 + D_6x_2^2, \\
E &= E_1x_0^2 + E_2x_0x_1 + E_3x_0x_2 + E_4x_1^2 + E_5x_1x_2 + E_6x_2^2, \\
F &= F_1x_0^2 + F_2x_0x_1 + F_3x_0x_2 + F_4x_1^2 + F_5x_1x_2 + F_6x_2^2. 
\end{align*}
Suppose that the coefficients of these quadratic forms satisfy:
\begin{enumerate}
\item $A \equiv 1 \bmod 8$, and $v_2(A_i) \geq 3$ for $i \neq 1$.
\item $v_2(B_1) = 0$, and $v_2(B_i) \geq 1 \text{ for } i \neq 1$.
\item $v_2(C_6) = 0$, and $v_2(C_i) \geq 1 \text{ for } i \neq 6$.
\item $D_4 \equiv 1 \bmod 8$, and $v_2(D_i) \geq 3$ for $i \neq 4$.
\item $v_2(E_4) = 0$, and $v_2(E_i) \geq 1 \text{ for } i \neq 4$.
\item $F_6 \equiv 1 \bmod 8$, and $v_2(F_i) \geq 3$ for $i \neq 6$.
\end{enumerate}
Then, for every $2$-adic point $P$ of $X$, we have $\inv_2 \calA(P) = 0$.
\qed
\end{lemma}

%****************************************************************************
\section{An example}
\label{s: an example}

Let $W \subset \Proj\bQ[x_0,x_1,x_2] \times \Proj\bQ[y_0,y_1,y_2]$ be the type $(2,2)$ divisor given by the vanishing of the bihomogeneous polynomial
\begin{equation}
\label{eq:2-2 divisor}
\begin{split}
-7x_0^2y_0^2 &+ 3x_0^2y_0y_1 + 10x_0^2y_0y_2 - 16x_0^2y_1^2 + 4x_0^2y_1y_2 - 40x_0^2y_2^2 - 16x_0x_1y_0^2 \\ 
			 &+ 4x_0x_1y_0y_2 + 8x_0x_1y_1^2 + 32x_0x_1y_2^2 + 16x_0x_2y_0^2 + 2x_0x_2y_0y_1 + 4x_0x_2y_0y_2 \\
			 &- 4x_0x_2y_1y_2 - 24x_1^2y_0^2 + 2x_1^2y_0y_1 + 4x_1^2y_0y_2 - 23x_1^2y_1^2 + 11x_1^2y_1y_2 \\
			 &- 40x_1^2y_2^2 + 8x_1x_2y_0^2 - 4x_1x_2y_0y_1 - 2x_1x_2y_0y_2 + 8x_1x_2y_1^2 - 4x_1x_2y_1y_2 \\
			 &- 8x_1x_2y_2^2 - 16x_2^2y_0^2 + 4x_2^2y_0y_1 + x_2^2y_0y_2 - 40x_2^2y_1^2 + 6x_2^2y_1y_2 - 23x_2^2y_2^2.
\end{split}
\end{equation}
As in~\S\ref{s: unramified}, the projections $\pi_i\colon W \to \bP^2$ give conic bundle structures on $W$ ramified over plane sextics $C_i$, $i = 1, 2$.  Consider the quadrics
\begin{equation}
\label{eq:quadrics}
\begin{split}
A &:= -7x_0^2 -16x_0x_1 + 16x_0x_2 -24x_1^2 + 8x_1x_2 -16x_2^2 \\
B &:= 3x_0^2 + 2x_0x_2 + 2x_1^2 - 4x_1x_2 + 4x_2^2 \\
C &:= 10x_0^2 + 4x_0x_1 + 4x_0x_2 + 4x_1^2 - 2x_1x_2 + x_2^2 \\
D &:= -16x_0^2 + 8x_0x_1 - 23x_1^2 + 8x_1x_2 - 40x_2^2 \\
E &:= 4x_0^2 - 4x_0x_2 + 11x_1^2 - 4x_1x_2 + 6x_2^2 \\
F &:= -40x_0^2 + 32x_0x_1 - 40x_1^2 - 8x_1x_2 - 23x_2^2.
\end{split}
\end{equation}
An equation for $C_1$ is then given by $-\frac{1}{2}\det(M) = 0$, with $M$ as in~\eqref{eq: K3}. An equation for $C_2$ can be found analogously.  The Jacobian criterion shows that both $C_1$ and $C_2$ are smooth; thus, for $i = 1, 2$, the double cover $X_i \to \bP^2$ ramified along $C_i$ is a K3 surface of degree $2$.

%****************************************************************************
\subsection{Primes of bad reduction}
\label{ss: bad red}
%To determine the local invariants of the algebra $\calA$ of $X_1$ given by~\eqref{eq:quaternion algebra}, we need to know the places of bad reduction of $X_1$.
A Groebner basis calculation over $\bZ$ shows that the primes of bad reduction of $X_1$ divide
\[
\begin{split}
m := \ & 1115508232640214856843363784231663793779083264535962688555888430968 \\
       & 8933364438401787008291918987282105867611490800785997644322303281186 \\
       & 8922614222749465991103128446037422257623280138072129654879995620391 \\
       & 0907629715637695773281604080143775185215794393627484442538367517916 \\
       & 8651952191024387026109016400178074232186309443422761817391984342483 \\
       &34511814400.
\end{split}
\]
Standard factorization methods quickly reveal a few small prime power factors of $m$:
\[
m = 2^8\cdot 5^2\cdot 7\cdot 89\cdot 173\cdot 257^2\cdot 263\cdot 650779^2\cdot m'.
\]
The remaining factor $m'$ has 318 decimal digits.  Factoring $m'$ with present day mathematical and computational technology is a difficult problem.  However, the presence of the second K3 surface $X_2$ supplies a backdoor solution: by Lemma~\ref{lem:W smooth}, a prime of bad reduction for $W$ is a prime of bad reduction for \emph{both} $X_1$ and $X_2$.  

Another Groebner basis calculation shows that the primes of bad reduction of $X_2$ divide
\[
\begin{split}
n := \ & 18468445386704774116897512713438756322646374324269134481315634355660 \\
       & 59216198653927410468599212130905398491499555534045930594495263034981 \\
       & 50100881353352665095649631677613412079293044973446406764509694053112 \\
       & 10471631439070548340358668493117334582314574674926223315439909955021 \\
       & 6973495867514854209929544319382116616140800
\end{split}
\]
Again, standard factorization methods give a few small prime power factors of $n$:
\[
n = 2^{11}\cdot 5^2\cdot 7\cdot 89\cdot 173\cdot 263\cdot 461^2\cdot 6547^2\cdot n',
\]
where $n'$ has 290 decimal digits.  Our observation says that we may reasonably expect that $m'$ and $n'$ have a large greatest common divisor (which is easily calculated using the Euclidean algorithm).   This is indeed the case:
\[
\begin{split}
\gcd(m',n') := \ & 809147864157687938441948148614369785987783654943839689121548451 \\
                 & 788111145202992792430023470932052297439515068068797124401938255 \\
                 & 799311490342451172887433057574480263654457987109316488649107.
\end{split}
\]
Here a small miracle happens: $\gcd(m',n')$ \emph{is a prime number!} This claim is rigorously verified using elliptic curve primality proving algorithms~\cite{AM}, implemented in both SAGE and {\tt magma}. We are now in a position to complete the factorization of $m$, and hence compute the primes of bad reduction for $X_1$, which are:
\[
\begin{split}
&2, 5, 7, 89, 173, 257, 263, 650779, \\
&521219738678096220868573969913582546660848099260319499224599922739, \\
&\gcd(m',n').
\end{split}
\]

\begin{remark}
Our numerical experiments yield several ``viable'' pairs $(X_1,\calA)$ that could be counter-examples to the Hasse principle explained by a transcendental Brauer-Manin obstruction arising from $\calA$, in the following sense: $X_1$ has geometric Picard rank $1$, and we can control the real and $2$-adic invariants of $\calA$ (using Corollary~\ref{cor:real place} and Lemma~\ref{lem:2-adic invariants}).  Out of a dozen or so viable candidates that our initial search yielded, the example we present is the only one we found for which $\gcd(m',n')$ is a prime number.  One can obtain further examples by computing $2$-adic invariants by ``brute force'' instead of using Lemma~\ref{lem:2-adic invariants}.
\end{remark}

%****************************************************************************
\subsection{Local points}
\label{ss: Local points}

By the Weil Conjectures, if $p > 22$ is a prime such that $X_1$ has smooth reduction $(X_1)_p$ at $p$, then $(X_1)_p$ has a smooth $\bF_p$-point, which can be lifted by Hensel's lemma to a smooth $\bQ_p$-point.  Thus, to show $X_1$ is locally soluble, it suffices to verify that $X_1$ has local points at $\bR$ (clear), and at $\bQ_p$ for primes $p \leq 19$ and primes $p > 19$ where $X_1$ has bad reduction.  This is indeed the case: we substitute integers  with small absolute value for $x_0$, $x_1$, and $x_2$, and check if $-\frac{1}{2}\det(M)$ is a square in $\bQ_p$.  The results are recorded in Table~\ref{ta:local points}.

\begin{table}
$$
\begin{tabular}{c|r|r|r|r}
$p$ & $x_0$ & $x_1$ & $x_2$ & $-\frac{1}{2}\det(M)$ \\
\hline
$2$ & $0$ & ${0}$ & ${-1}$ & $57872$ \\
\hline
$3$ & ${-1}$ & ${-1}$ & ${1}$ & $1622952$ \\
\hline
$5$ & ${-1}$ & ${-1}$ & ${-1}$ & $736256$ \\
\hline
$7$ & ${-1}$ & ${-1}$ & $0$ & $256575$ \\
\hline
$11$ & ${-1}$ & ${-1}$ & ${-1}$ & $736256$ \\
\hline
$13$ & ${-1}$ & ${-1}$ & ${-1}$ & $736256$ \\
\hline
$17$ & ${-1}$ & ${-1}$ & $1$ & $1622952$ \\
\hline
$19$ & ${-1}$ & ${-1}$ & ${-1}$ & $736256$ \\
\hline
$89$ & ${-1}$ & ${0}$ & ${-1}$ & $80019$ \\
\hline
$173$ & ${-1}$ & ${-1}$ & $0$ & $256575$ \\
\hline
$257$ & ${-1}$ & ${-1}$ & ${-1}$ & $736256$ \\
\hline
$263$ & ${-1}$ & ${-1}$ & $0$ & $256575$ \\
\hline
$650779$ & ${-1}$ & ${-1}$ & $1$ & $1622952$ \\
\hline
\multirow{3}{*}{}$5212197386780962208687$ & & & & \\
				$3969913582546660848099$  & ${-1}$ & ${-1}$ & ${-1}$ & $736256$ \\
				$260319499224599922739$  & & & & \\
\hline
$\gcd(m',n')$ & ${-1}$ & ${-1}$ & ${-1}$ & $736256$ \\
\end{tabular}
$$
\caption{Verifying $X_1$ has $\bQ_p$-points at small $p$ and primes of bad reduction.}
\label{ta:local points}
\end{table}

%****************************************************************************
\subsection{Picard Rank 1}

In this section we show $X_1$ has (geometric) Picard rank  1.  This will allow us to conclude that the obstruction to the Hasse principle arising from $\calA$ is genuinely transcendental.  Until recently, the method to prove a K3 surface has odd Picard rank, devised by van Luijk and refined by Kloosterman, and Elsenhans and Jahnel~\cite{vanLuijk,Kloosterman,ElsenhansJahnelRefinement}, required point counting over extensions of the residue field at two primes of good reduction.  A recent result of Elsenhans and Jahnel allows us to prove odd Picard rank using information at two primes, but counting points over extensions of a single residue field.

\begin{theorem}[{\cite{ElsenhansJahnelOnePrime}}]
\label{thm:torsion-free cokernel}
Let $f\colon X \to \Spec \bZ$ be a proper, flat morphism of schemes. Suppose there is a rational prime $p \neq 2$ such that the fiber $X_p$ of $f$ at $p$ satisfies $H^1(X_p,\sO_{X_p}) = 0$. Then the specialization homomorphism $\Pic(X_{\overline{\bQ}}) \to \Pic(X_{\overline{\bF}_p})$ has torsion-free cokernel.
\end{theorem}

We deduce the following generalization of~\cite[Example~1.6]{ElsenhansJahnelOnePrime}.

\begin{proposition}
\label{prop:check rho = 1}
Let $X$ be a K3 surface of degree $2$ over $\bQ$, given as a double cover $\pi\colon X \to \bP^2$ ramified over a smooth plane sextic curve $C$. Let $p$ and $p'$ denote two odd primes of good reduction for $X$. Assume that there exists a line $\ell$ that is tritangent to the curve $C_p$, and suppose further that $\Pic(\overline{X}_p)$ has rank $2$ and is generated by the curves in $\pi_p^{-1}(\ell)$. If there are no tritangent lines to the curve $C_{p'}$, then $\Pic(\overline{X})$ has rank $1$.
\end{proposition}

\begin{proof}
Since $\Pic(\overline{X})$ injects into $\Pic(\overline{X}_p)$, if $\Pic(\overline{X})$ has rank $2$, then the tritangent line $\ell$ must lift to a tritangent line $L$ in characteristic $0$, by Theorem~\ref{thm:torsion-free cokernel}. Degree considerations show that $L$ cannot break upon reduction modulo $p'$. This contradicts the assumption that the curve $C_{p'}$ has no tritangent lines.
\end{proof}

\begin{remarks}
\label{rems:computational dividends}
Proposition~\ref{prop:check rho = 1} is computationally useful because:
\begin{enumerate}
\item Checking the existence of a tritangent line modulo $p'$ is an easy Groebner basis calculation; see~\cite[Algorithm~8]{ElsenhansJahnelrankone}.
\item Given a K3 surface of degree $2$ over $\bQ$, we can quickly \emph{search} for small primes $p$ of good reduction over which the branch curve $C_{p'}$ of the double cover $X_{p'}\to \bP^2_{\bF_{p'}}$ has a tritangent line.
\end{enumerate}
\end{remarks}

%2x_1^2(x_0^2 + 2x_0x_1 + 2x_1^2)^2 + 2(x_0 + 2x_2)(x_0^5 + x_0^4x_1 + x_0^3x_1x_2 + x_0^2x_1^3 + x_0^2x_1^2x_2
% + 2x_0^2x_2^3 + x_0x_1^4 + 2x_0x_1^3x_2 + x_0x_1^2x_2^2 + x_1^5 + 2x_1^4x_2 + 2x_1^3x_2^2 + 2x_2^5)

Our particular surface $X_1$ reduces modulo $3$ to the (smooth) K3 surface
\[
\begin{split}
w^2 &= 2x_1^2(x_0^2 + 2x_0x_1 + 2x_1^2)^2 + (2x_0 + x_2)(x_0^5 + x_0^4x_1 + x_0^3x_1x_2 + x_0^2x_1^3 + x_0^2x_1^2x_2 \\
\quad &+ + 2x_0^2x_2^3 + x_0x_1^4 + 2x_0x_1^3x_2 + x_0x_1^2x_2^2 + x_1^5 + 2x_1^4x_2 + 2x_1^3x_2^2 + 2x_2^5)
\end{split}
\]
From the expression on the right hand side, it is clear that $2x_0 + x_2 = 0$ is a tritangent line to the branch curve of the double cover. The components of the pullback of this line generate a rank $2$ sublattice of $\Pic\big( (\overline{X_1})_3\big)$
Let $N_n := \#X_1(\bF_{3^n})$; counting points we find
$$
\begin{array}{r|r|r|r|r|r|r|r|r|r}
N_1 & N_2 & N_3 & N_4 & N_5 & N_6 & N_7 & N_8 & N_9 & N_{10}  \cr
\hline
7 & 79 & 703 & 6607 & 60427 & 532711 & 4792690 & 43068511 & 387466417 & 3486842479 \cr
\end{array}
$$
This is enough information to determine the characteristic polynomial $f$ of Frobenius on $H^2((X_1)_{\overline{\bF}_3},\overline{\bQ}_\ell)$; see, for example~\cite{vanLuijk} (the sign of the functional equation for $f$ is negative---a positive sign gives rise to roots of $f$ of absolute value $\neq 3$).
 Setting $f_3(t) = 3^{-22}f(3t)$, we obtain a factorization into irreducible factors as follows:
\[
\begin{split}
f_3(t) = \frac{1}{3}(t - 1)(t + 1)(3t^{20} &+ 3t^{19} + 5t^{18} + 5t^{17} + 6t^{16} + 2t^{15} + 2t^{14} - 3t^{13} - 4t^{12} - 8t^{11}\\
										 &- 6t^{10} - 8t^9 - 4t^8 - 3t^7 + 2t^6 + 2t^5 + 6t^4 + 5t^3 + 5t^2 + 3t + 3)
\end{split}
\]
The number of roots of $f_3(t)$ that are roots of unity give an upper bound for $\Pic((X_1)_{\overline{\bF}_3})$.  The roots of the degree $20$ factor of $f_{3}(t)$ are not integral, so they are not roots of unity. We conclude that $\rk \Pic((X_1)_{\overline{\bF}_3}) = 2$.

A computation shows that $X_1$ has no line tritangent to the branch curve when we reduce modulo $p' = 11$ (see Remark~\ref{rems:computational dividends}(i)).  Note that the surface is not smooth at $p' = 5, 7$. Applying Proposition~\ref{prop:check rho = 1}, we  obtain:

\begin{proposition}
\label{prop: Picard Rank 1}
The surface $X_1$ has geometric Picard rank $1$.
\end{proposition}

%****************************************************************************
\subsection{Local invariants}

In this section we compute the local invariants of the algebra $\calA$ for our particular surface $X_1$.  

\begin{proposition}
\label{prop: Invariants computed}
Let $p\leq \infty$ be a prime number. For any $P \in X_1(\bQ_p)$, we have
\[
\inv_p\big(\calA(P)\big) = 
\begin{cases}
0, & \textup{if } \bQ_p \neq \bR, \\
1/2, & \textup{if } \bQ_p = \bR.
\end{cases}
\]
\end{proposition}

\begin{proof}
Whenever $p \neq 2$ is a finite prime of good reduction for $X_1$, we have $\inv_p\big(\calA(P)\big) = 0$ for all $P$, by Lemma~\ref{lem:most invariants}.  

At every odd prime of bad reduction of $X_1$, the singular locus consists of $r < 8$ ordinary double points: for most of these primes $p$ the claim follows because the valuation at $p$ of the discriminant of $X_1$ is one, by our work in~\S\ref{ss: bad red}, so the singular locus consists of a single ordinary double point. For the remaining primes, a straightforward computer calculation does the job.

Together with  Proposition~\ref{prop:constant eval map} and Lemma~\ref{lem:no cyclic covers}, this implies that $\inv_p\big(\calA(P)\big)$ is independent of $P$; it thus suffices to evaluate these invariants at a single point $P$. We use the local points listed in Table~\ref{ta:local points} to verify that all the local invariants vanish.  

Finally, the quadrics~\eqref{eq:quadrics} are readily seen to satisfy the hypotheses of Lemmas~\ref{lem:real place} and~\ref{lem:2-adic invariants}, which establishes the claim for real and $2$-adic points of $X_1$, using Corollary~\ref{cor:real place}.
\end{proof}

%****************************************************************************
\subsection{Proof of Theorem~\ref{thm: main}}

The first part of the Theorem is just Proposition~\ref{prop: brauer elt}. We specialize now to the case $k = \bQ$.

Let $A,\dots,F$ be as in~\eqref{eq:quadrics}, so that $X$ is the surface $X_1$ considered throughout this section. The cohomology group $H^1\big(\bQ,\Pic(\overline{X})\big)$ is trivial, because $\Pic(\overline{X}) \cong \bZ$, with trivial Galois action, by Proposition~\ref{prop: Picard Rank 1}. By~\eqref{eq: HS}, we have $\Br_1(X) = \Br_0(X)$. Hence, the class $\calA \in \Br(X)$ is transcendental, if it is not constant.

We established in \S\ref{ss: Local points} that $X(\bA) \neq \emptyset$. On the other hand, $X(\bA)^\calA = \emptyset$, by Proposition~\ref{prop: Invariants computed}. This shows that $\calA$ is nonconstant, and that $X(\bA)^{\Br} = \emptyset$.
\qed

%****************************************************************************
%\subsection{The isogenous K3 surface}
%Recall the construction of $X_1$ also gives rise to a second K3 surface $X_2$, with its own transcendental 
%
%\tony{this section needs to be expanded...I had forgotten about it.}

%****************************************************************************
\section{Computations}
\label{s: computations}

In the interest of transparency, we briefly outline the computations that led to the example witnessing the second part of Theorem~\ref{thm: main}. The basic idea is to construct ``random'' K3 surfaces of the form~\eqref{eq: explicit K3}, and perform a series of tests that guarantee the statement of Theorem~\ref{thm: main} holds.  Any surface left over after Step~7 below is a witness to this theorem.  \\ %The computations were performed in \tony{MAGMA, M2 and Sage} \\

\noindent {\bf Step 1: Seed polynomials.} Generate random homogenous quadratic polynomials 
\[
A, B, C, D, E, \textup{ and } F \in \bZ[x_0,x_1,x_2],
\]
with coefficients in a suitable range, subject to the constraints imposed by the hypotheses of Lemma~\ref{lem:2-adic invariants}. We also require that the signs of $x_0^2$, $x_1^2$ and $x_2^2$ are positive for $B$, $C$ and $E$, and negative for $A$, $D$ and $F$, to improve the chances that the hypotheses of Lemma~\ref{lem:real place} are satisfied. If these hypotheses are not satisfied, then start over.\\

\noindent {\bf Step 2: Smoothness.} Compute $f := -\frac{1}{2}\det(M)$, where $M$ is the matrix in~\eqref{eq: K3}. This is an equation for the curve $C_1$. Use the Jacobian criterion to check smoothness of $C_1$ over $\bQ$ and $\bF_3$ (the latter will be needed to certify that the K3 surface $X_1$ has Picard rank $1$). If either condition is not satisfied, then start over.\\

\noindent {\bf Step 3: Tritangent lines.} Here we have the hypotheses of Proposition~\ref{prop:check rho = 1} in mind. Over $\bF_3$, use \cite[Algorithm~8]{ElsenhansJahnelrankone} to test for the existence of a tritangent line to $C_1$. Let
\[
S := \{p : 5 \leq p \leq 100 \textup{ a prime of good reduction for $C_1$}\}.
\]
Find $p \in S$, such that $C_1$ over $\bF_p$ has no tritangent line. If either test fails, then start over. \\

\noindent {\bf Step 4: Local points.} For primes $p \leq 22$ and $p = \infty$, test for $\bQ_p$-points of $X_1/\bQ : {w^2 = f}$ by plugging in integers with small absolute value (typically $1$ or $2$) for $x_0$, $x_1$ and $x_2$, and determining whether $f$ is a $p$-adic square. If this test fails,  then it is plausible that $X_1$ has no local points (false negatives are certainly possible); start over. \\

\noindent {\bf Step 5: Point Counting.} Use \cite[Algorithm~15]{ElsenhansJahnelrankone} to determine $X_1(\bF_{3^n})$ for $n = 1,\dots,10$. This algorithm counts \emph{Galois orbits} of points, saving a factor of $n$ when counting $\bF_{3^n}$-points. Use \cite[Algorithms~21 and~23]{ElsenhansJahnelrankone} to determine an upper bound $\rho_{\textup{up}}$ for the geometric Picard number of the surface $X_1$ over $\bF_3$. If $\rho_{\textup{up}} > 2$, then start over. Otherwise, Proposition~\ref{prop:check rho = 1} guarantees that $X_1$ has geometric Picard number 1, by our work in Step 3.\\

\noindent {\bf Step 6: Primes of bad reduction.} The primes of bad reduction of $X_1$ and $C_1$ coincide. The latter divide the generator $m$ of the ideal obtained by saturating
\[
\left\langle f, \frac{\partial f}{\partial x_0}, \frac{\partial f}{\partial x_1}, \frac{\partial f}{\partial x_2} \right\rangle \subset \bZ[x_0,x_1,x_2]
\] 
by the irrelevant ideal and eliminating $x_0$, $x_1$ and $x_2$. Compute an equation for $C_2$, as well as the analogous integer $n$ giving its primes of bad reduction. Typically, $m$ and $n$ will be very large. Proceed as in \S\ref{ss: bad red} to factorize them. \\

\noindent {\bf Step 7: Computations at places of bad reduction.} At odd places of bad reduction, check for local points, as in Step~4. Determine the (geometric) singular locus. If at any prime in question the locus does not consist of $r < 8$ ordinary double points, then start over. Use to the local points found to compute the (constant) value the invariants of $\calA$ takes at these places.

%****************************************************************************

\bibliographystyle{alpha}
\bibliography{K3Hasse}
\nocite{*}

\end{document}